\documentclass[reqno]{amsart} 
\usepackage{amsmath}
\usepackage{amsfonts}
\usepackage{amssymb}
\usepackage{amsthm}
\usepackage{mathtools}
\usepackage{nccmath}
\usepackage[initials]{amsrefs}

\usepackage{xcolor}
\usepackage[colorlinks,allcolors=blue]{hyperref} 

\newtheorem{theorem}{Theorem}
\newtheorem{proposition}[theorem]{Proposition}
\newtheorem{lemma}[theorem]{Lemma}
\newtheorem{corollary}[theorem]{Corollary}

\theoremstyle{definition}
\newtheorem{example}[theorem]{Example}

\newcommand{\idealztn}[2]{#1\mathop{(\! + \! )}#2}
\newcommand{\ideal}[1]{\left< #1 \right>}

\begin{document}

\author{S. Eftekhari, S.H. Jafari and M.R. Khorsandi$^{*}$} 

\title[Factorization of idealization in commutative rings]
{Some factorization properties of idealization in commutative rings with zero divisors}
\thanks{$^*$Corresponding author}
\thanks{E-mail addresses: sina.eftexari@gmail.com, shjafari@shahroodut.ac.ir and \\ 
khorsandi@shahroodut.ac.ir}

\maketitle

\begin{center}
{\it Faculty of Mathematical Sciences, Shahrood University of Technology,\\
P.O. Box 36199-95161, Shahrood, Iran.}
\end{center}

\begin{abstract}
We study some factorization properties of the idealization $R \mathop{(\! + \! )} M$ of a module $M$ in a commutative ring $R$ which is not necessarily a domain. We show that $R \mathop{(\! + \! )} M$ is ACCP if and only if $R$ is ACCP and $M$ satisfies ACC on its cyclic submodules. We give an example to show that the BF property is not necessarily preserved in idealization, and give some conditions under which $R \mathop{(\! + \! )} M$ is a BFR. We also characterize the idealization rings which are UFRs.
\end{abstract}

\maketitle

\section{Introduction}
Throughout this paper, all rings are commutative with identity and all modules are unital.  Anderson and Valdes-Leon \cite{anderson1996factorization} provided a framework for studying factorization in commutative rings which are not necessarily domains. One of the important constructions in commutative algebra which always results in rings with nontrivial zero divisors is the idealization of a (nonzero) module. Let $R$ be a ring and $M$ be an $R$-module. The set $R \times M$ with the multiplication $(r_1,x_1)(r_2,x_2) = (r_1r_2, r_1x_2 + r_2x_1)$ and with natural addition is a ring called the \textit{idealization} of $M$ in $R$, and is denoted by $\idealztn{R}{M}$. This construction can also be viewed in matrix form as $\left\{ \begin{psmallmatrix} r & x \\ 0 & r \end{psmallmatrix} \mid r \in R , x \in M \right\}$. If $M = R$, the ring $\idealztn{R}{R}$, which is called the \textit{self-idealization} of $R$, can also be viewed as the ring $R[X]/\left< X^2 \right>$. For an introduction to idealization and its properties we refer the reader to \cite[Section~25]{huckaba1988commutative}. Also \cite{anderson2009idealization}, besides providing new results, is a great survey on this topic. In fact, \cite[Section~5]{anderson2009idealization} is about the factorization properties of rings that are constructed by idealization.

We recall some of the basic properties of idealization \cite[Theorem~25.1]{huckaba1988commutative}. An element $(r, x) \in \idealztn{R}{M}$ is a unit if and only if $r$ is a unit in $R$. A subset of $\idealztn{R}{M}$ of the form $\idealztn{I}{N}$ is an ideal if and only if $I$ is an ideal of $R$, $N$ is a submodule of $M$, and $IM \subseteq N$. Moreover, a prime ideal of $\idealztn{R}{M}$ is of the form $\idealztn{\mathfrak{p}}{M}$ where $\mathfrak{p}$ is a prime ideal in $R$. Similarly, maximal ideals of $\idealztn{R}{M}$ has the form $\idealztn{\mathfrak{m}}{M}$ for some maximal ideal $\mathfrak{m}$ of $R$. Finally, for ideals $\idealztn{I_i}{N_i}$ ($i = 1,2$), we have $(\idealztn{I_1}{N_1})(\idealztn{I_2}{N_2}) = \idealztn{I_1I_2}{(I_1N_2 + I_2N_1)}$.

A ring $R$ is called ACCP when every ascending chain of principal ideals of $R$ stabilizes. Similarly, a module with ACC on its cyclic submodules is called ACCC. A ring $R$ is called a \emph{bounded factorization ring} or BFR if for every nonzero nonunit element $r \in R$ there exists an $n_r \in \mathbb{N}$ such that if $r = r_1 \dotsm r_n$ for nonunit $r_i$'s, then $n \leq n_r$. Similarly, an $R$-module $M$ is called a BFM (or a BF R-module) if for every nonzero $x \in M$, there exists $n_x \in \mathbb{N}$ such that if $x= r_1 \dotsm r_n y$, where $r_i$'s are nonunits in $R$ and $y \in M$, then $n \leq n_x$. We can also say that a ring (or a module, or an element) has the BF property. Also, if $R$ is a domain, we usually use the abbreviation BFD instead of BFR.

There are a number of ways to generalize the notion of irreducibility to commutative rings with zero divisors (see \cite[Section~2]{anderson1996factorization}). We only need the following definitions: We call two elements $a, b \in R$ \emph{associates} and write $a \sim b$ if $\ideal{a} = \ideal{b}$, and a nonunit element $a \in R$ is called \emph{irreducible} or an \emph{atom} if when $a = bc$, then either $a \sim b$ or $a \sim c$. A ring $R$ is called \emph{pr\'{e}simplifiable} (introduced by Bouvier \cite{Bouvier1974}) if whenever $a=ab$, then either $a=0$ or $b$ is a unit. These rings are quite important in factorization theory since many factorization properties of domains also hold in them (see \cite{anderson1996factorization}). Any BFR is pr\'{e}simplifiable (see \cite[p.~456]{anderson1996factorization}), but an ACCP ring is not necessarily pr\'{e}simplifiable (consider any ACCP ring with a nontrivial idempotent, like $F \times F$ where $F$ is a field). A ring $R$ is called a UFR if any nonzero nonunit element of $R$ can be written as a product of atoms and this factorization is unique up to order and associates. Obviously, any UFR is a BFR, and any BFR is ACCP.

Current results on the factorization properties of $\idealztn{R}{M}$ are mostly for the case where $R$ is a domain (see \cite{agargun2001factorization, anderson1996factorization, anderson2009idealization, axtell2002u, axtell2017factorizations, chang2013factorization}). The goal of this paper is to generalize some of these results to arbitrary commutative rings.

First, we show that for any ring $R$ and $R$-module $M$, the ring $\idealztn{R}{M}$ is ACCP if and only if $R$ is ACCP and $M$ is ACCC. Then, we give an example to show that the BF property is not necessarily preserved in the idealization and we provide some sufficient conditions under which $\idealztn{R}{M}$ becomes a BFR. Finally, we characterize the rings $\idealztn{R}{M}$ which are UFRs.

\section{Results}

When $R$ is a domain, $\idealztn{R}{M}$ is ACCP if and only if $R$ is ACCP and $M$ is ACCC (\cite[Theorem~5.2(2)]{anderson1996factorization}). We show that this is also the case for arbitrary rings  though we need a completely different approach.

A quotient module of an ACCC module is not necessarily ACCC. However, in the next lemma we show that certain chains of cyclic submodules in a quotient module of an ACCC module stabilize. The idea is similar to what Frohn has done in the proof of \cite[Lemma~1]{frohn2004accp}.

\begin{lemma}
\label{lemma1}
Let $M$ be an $R$-module which is ACCC and let N be an $R$-submodule of $M$. Also, suppose that $R\overline{x_1} \subseteq R\overline{x_2} \subseteq \dotsb$ is an ascending chain of cyclic submodules in $M/N$, and $\overline{x_n} = r_{n}\overline{x_{n+1}}$ for some $r_n \in R$. If $N = r_iN$ for every $i$, then the chain $R\overline{x_1} \subseteq R\overline{x_2} \subseteq \dotsb$ stabilizes.
\end{lemma}
\begin{proof}

First, we note that
\begin{align*}
x_1 &= r_1x_2 + y_1 \quad (\text{for some} \; y_1\in N) \\
&=r_1(x_2 + y_1/r_1),
\end{align*}
where by $y_1/r_1$ we mean a fixed element in $N$ such that $r_1(y_1/r_1) = y_1$.

Now
\begin{align*}
x_2 + y_1/r_1 &= r_2x_3 + y_2 + y_1/r_1 \quad (\text{for some} \; y_2\in N) \\
&= r_2(x_3 + y_2/r_2 + (y_1/r_1)/r_2).
\end{align*}
Proceeding this way, we get a chain
$
Rx_1 \subseteq R(x_2 + z_2) \subseteq R(x_3 + z_3) \subseteq \dotsb
$
where $z_i \in N$. Since $M$ is ACCC, we have
$
R(x_k + z_k) = R(x_{k+1} + z_{k+1}) = \dotsb
$
for some $k$.
Hence, for every $n \geq k$, $ x_{n+1} + z_{n+1} = s_nx_n + s_n z_n$ for some $s_n \in R$, and so
\begin{align*}
x_{n+1} &= s_nx_n + s_n z_n - z_{n+1}
\in s_nx_n + N .
\end{align*}
Therefore, the chain $R\overline{x_1} \subseteq R\overline{x_2} \subseteq \dotsb$ stabilizes.
\end{proof}

\begin{theorem}
\label{ACCPidealztn}
For any ring $R$ and $R$-module $M$, $\idealztn{R}{M}$ is ACCP if and only if $R$ is ACCP and $M$ is ACCC.
\end{theorem}
\begin{proof}

$(\Rightarrow)$
The proof is the same as the one for the domain case (\cite[Theorem~5.2(2)]{anderson1996factorization}).

$(\Leftarrow)$
Set $T:= \idealztn{R}{M}$ and let
$
(r_1,x_1)T \subseteq (r_2,x_2)T \subseteq \dotsb
$
be an ascending sequence of principal ideals of $T$. Also, suppose that
$ (r_i, x_i) = (s_i, y_i) (r_{i+1},x_{i+1}) $ for some $(s_i, y_i) \in T$.
Since $R$ is ACCP, there exists a $k \in \mathbb{N}$ such that $r_kR = r_{k+1}R = \dotsb$. Set $I := r_kR$ and $N := r_kM$. It is easy to see that, for every $n \geq k$,
\begin{equation}
\label{eq1}
s_{n} N = N, \; \text{and} \; s_{n} I = I,
\end{equation}
and
\begin{equation}
\label{eq2}
x_n = s_n x_{n+1} + r_{n+1}y_n \in s_n x_{n+1} + N.
\end{equation}

By (\ref{eq1}) and Lemma \ref{lemma1}, for some $m \in \mathbb{N}$, 
\begin{equation}
\label{eq4}
N + Rx_m = N + R x_{m+1} = \dotsb .
\end{equation}
 Also, by (\ref{eq1}) and (\ref{eq2}),
$x_n \in s_n (Rx_{n+1} + N) = s_n (Rx_n + N) $,
and so for every $n \geq \max(k, m)$
\begin{equation}
\label{eq3}
N+ Rx_n = s_n (N + Rx_n) .
\end{equation}

Without loss of generality and by way of contradiction, we may assume that
$
(r_1,x_1)T \subsetneq (r_2,x_2)T \subsetneq \dotsb
$
and that the equations (\ref{eq1}), (\ref{eq4}) and (\ref{eq3}) hold for any $n, m \in \mathbb{N}$.

By (\ref{eq1}), there exists a $v \in R$ such that $s_1vr_2 = r_2$, so $(1-s_1v)r_2= 0$ and hence $(1-s_1v)N = 0$ . On the other hand, by (\ref{eq4}) and (\ref{eq3}), in the $R$-module $M/N$, $Rs_1\overline{x_2} = R\overline{x_2}$ and so for some $u \in R$ we have  $(1 - s_1u)x_2 \in N$, and so $(1-s_1v)(1 - s_1u)x_2 = 0$. Since  
\[(1-s_1v)(1-s_1u) = (1 - s_1((u+v) - s_1uv)) ,\] for $w = (u+v) - s_1uv$, we have $(1-s_1w)r_2 = 0$ and $( 1-s_1w)x_2 = 0$.

Therefore
\begin{align*}
\left[ (1,0) - (s_1 , y_1)(s_1w^2, - w^2 y_1) \right] (r_2, x_2)
& = (1 - s_1^2w^2, 0)(r_2, x_2) \\
& = (1+s_1w, 0)( 1-s_1w, 0)(r_2, x_2) \\
& = 0 ,
\end{align*}
and so
\begin{align*}
(r_2, x_2) &= (s_1w^2, -w^2 y_1)(s_1 , y_1)(r_2, x_2) \\
& = (s_1w^2, -w^2 y_1)(r_1 , x_1),
\end{align*}
hence $(r_2, x_2) T \subseteq (r_1, x_1) T$ which is a contradiction.
\end{proof}

\begin{corollary}
Let $R$ be a ring. Then $R$ is ACCP if and only if $\idealztn{R}{R}$ is ACCP.
\end{corollary}

Next, we consider the bounded factorization property. When $R$ is a domain, the ring $\idealztn{R}{M}$ is a BFR if and only if $R$ is a BFD and $M$ is BF $R$-module (\cite[Theorem~5.2(3)]{anderson1996factorization}). First, we show that the ``if'' part of this result does not hold when $R$ is not a domain.

\begin{example}
Let $ S := \mathbb{Z}_2 \left[ \bigcup_{i \in \mathbb{N}} \{ X_{i, 1}, \dotsc ,X_{i, i+1} \} \right] $ and consider the following ideals in $S$:
\begin{align*}
I_1 &:= \left<  \{ X_{i,c_1} \dotsm X_{i, c_{i + 1}} \mid 1 \leq c_j \leq i+1, \; 1 \leq j \leq i + 1, \;  i \in \mathbb{N}  \} \right> \\
I_2 &:= \left< \{ X_{i,j}X_{k, l} \mid i,j,k,l \in \mathbb{N}, \, i \neq k, \, j \leq i + 1, \, l \leq k + 1
\} \right> \\
\intertext{and}
I_3 &:=\ \big< \big\{ \sum_{1 \leq j \leq i + 1} X_{i, 1} \dotsm \widehat{ X_{i , j}} \dotsm X_{i, i+1}  \\
&\qquad\qquad- \sum_{1 \leq j \leq i + 2} X_{i + 1, 1} \dotsm \widehat{ X_{i + 1 , j}} \dotsm X_{i + 1, i+2} \mid i, j \in \mathbb{N} \big\} \big> .
\end{align*}

We claim that the ring $R = \frac{S}{I_1 + I_2 + I_3}$ is a BFR, but the ring $\idealztn{R}{R}$ is not a BFR.

 Suppose on the contrary that there exist nonzero nonunit elements $f$ and $f_{i,j}$ such that $ f = f_{i,1} \dotsm f_{i, k_i} $, where $\{ k_i \}_{i \in \mathbb{N}}$ is not bounded. Let $x_{i,j}$ be the image of $X_{i, j}$ in $R$. Also, $\overline{F_{i,j}} = f_{i,j}$ and $\overline{F} = f$ for some $F, F_{i, j} \in S$. Moreover, we set $\beta_i := \sum_{1 \leq j \leq i + 1} X_{i, 1} \dotsm \widehat{ X_{i , j}} \dotsm X_{i, i+1}$.
 
 Suppose $m$ is the largest integer for which some $X_{m,t}$ appears in $F$. Let $v: S \rightarrow S$ be the homomorphism that sets any $X_{i,j}$ such that $i \geq m+1$ equal to $0$. We note that each nonunit in $R$ is actually nilpotent and so for any $A \in S$ with nonzero coefficient, $\overline{A}$ is a unit. We may consider each $F_{i,j}$ as a sum of proper subproducts of generators of $I_1$ since other terms become $0$ in $R$. If follows that for a large enough $i$, $v(F) = v(F_{i,1}) \dotsm v(F_{i, k_i}) = 0$ modulo $v(I)$. So, $F = v(F) \in v(I) \subseteq I_1 + I_2 + v(I_3)$. Now $v(I_3) \subseteq I_3 + v(S) \beta_m$, so $F - L\beta_m \in I$ for some $L \in v(S)$, and we may assume $L = 1$ since the product of any variable in $v(S)$ and $\beta_m$ is in $I_1 + I_2$. Therefore, $f = \overline{\beta_m}$. On the other hand, $\overline{\beta_1} = \overline{\beta_2} = \dotsb$ and so for every $i \in \mathbb{N}$, we have $f = \overline{\beta_i}$.

For every $i \in \mathbb{N}$, the element $\beta_i$ is irreducible in $S$ (and in any subring of $S$ resulting from removing some $X_{j,k}$ with $j \neq i$). One way to see this is to write $\beta_i$ as a polynomial with respect to the variable $X_{i,1}$, that is 
\[
\beta_i = X_{i,2} \dotsm X_{i, i+1} + \sum_{2 \leq j \leq i+1} \left( X_{i, 2} \dotsm \widehat{ X_{i , j}} \dotsm X_{i, i+1} \right) X_{i, 1} . 
\]
Since $S$ is a domain, the only way this polynomial can factor into two nonunits is that its coefficient have a nonunit common divisor. But $S$ is also a UFD and no $X_{i, j}$ ($2 \leq j \leq i+1$) divides  $\sum_{2 \leq j \leq i+1} \left( X_{i, 2} \dotsm \widehat{ X_{i , j}} \dotsm X_{i, i+1} \right)$. Hence, no such factorization is possible.

Now we show that $x_{1,1} + x_{1,2} = \overline{\beta_1} (=f)$ is irreducible in $R$. Suppose on the contrary that $f = gh$
where $g$ and $h$ are nonunit elements of $R$. Let $G, H \in S$ be such that $\overline{G} = g$ and $\overline{H} = h$. We get
\[
X_{1,1} + X_{1,2} - GH =  \sum_{1 \leq i \leq k} F_i (\beta_i - \beta_{i+1}) + F' + F''
\]
 
where $F_i \in S$, $F' \in I_1$, $F'' \in I_2$ (and, of course, $F_i = 0$ for every $i \gneq k$). For $A \in S$ let $\ell(A)$ denote the sum of the monomials of the least total degree and let $T(A)$ denote the total degree of $A$. We have $T(\ell(GH)) \geq 2$ since none of the elements $G$ and $H$ has a nonzero constant term for otherwise they would be units in $R$. So $\beta_1$ must appear on the right hand side and the only way for this is for $F_1$ to have $1$ as the constant term. After removing $\beta_1$ from each side we get
\[
 (-GH = )GH = \beta_2 + (F_1 - 1)(\beta_1 - \beta_2) + \sum_{2 \leq i \leq k} F_i (\beta_i - \beta_{i+1}) + F' + F'' .
\]
Set $X_{1,1} = X_{1,2}= 0$.  Now, either $(F_1 - 1)\beta_2 = 0$ or $T(\ell(F_1 - 1)\beta_2) \geq 3$; also, $T(\ell(F')) \geq 3$; so $\beta_2$ or any sum of its terms cannot appear in these parts. Moreover, we note that any monomial in $F''$ has a subproduct of the form $X_{i,j}X_{i',j'}$ where $i \neq i'$ and so this part of the right hand side also cannot contain $\beta_2$ or any sum of its terms. We cannot have $\ell(G)\ell(H) = \beta_2$ since either $(GH =) \ell(G)\ell(H) = 0$ or $0 \neq  \ell(G)\ell(H) = \beta_2$ which is not possible since as we saw $\beta_2$ is irreducible in $S$ (and this is also the case after setting $X_{1, 1}$ and $X_{1, 2}$ to $0$). Hence, the only remaining possibility is for $F_2$ to have a nonzero constant term. Now, we can repeat the argument until, eventually, we get the element $F_{k+1}$ with a nonzero constant term, and that is a contradiction. Therefore, $f = \beta_i$ is irreducible which of course means it has the BF property too. But this contradicts our initial assumption on $f$. Hence, $R$ is a BFR.
	
The ring $\idealztn{R}{R}$ is not a BFR, since for every $i \in \mathbb{N}$,
\[
(x_{1,1}, 1)(x_{1,2}, 1) = (0, \overline{\beta_1}) = (0, \overline{\beta_i}) =  (x_{i,1}, 1)(x_{i, 2}, 1) \dotsm (x_{i, i+1}, 1) .
\]
\end{example}

Let us call a factorization minimal if we cannot remove any factor from it. Formally, $x = a_1 \dotsm a_n$ is minimal if $x \neq \prod_{a_i \in A} a_i$ for any $A \subsetneq \{a_1, \dotsc , a_n \}$. Also, we say that an element $r \in R$ has the bounded minimal factorization property if the set
\begin{align*}
\{ n \!\in\! \mathbb{N} \,|\, \text{There exist nonunits }  a_1, \dotsc, a_n \in R  \text{ such that }  r \!=\! a_1 \dotsm a_n   \text{ is minimal} \}
\end{align*}
is bounded. For the element $0 \in R$, this property coincides with the notion of being \textit{U-bounded} which is defined in \cite[Section~4]{agargun2001factorization}, using the language of U-factorization. Although we avoided discussing U-factorizations in this paper, nevertheless we use the term U-bounded in this case.

\begin{proposition}
\label{BFRidealztn}
Let $R$ be a ring and let $M$ be an $R$-module.
\begin{enumerate}
\item
If $\idealztn{R}{M}$ is a BFR, then $R$ is BFR and $M$ is a BFM.
\item
Let $R$ be a BFR and $M$ a BFM. Also, suppose that $0 \in R$ is U-bounded. Then $\idealztn{R}{M}$ is a BFR.
\end{enumerate}
\end{proposition}
\begin{proof}

(1) This is similar to \cite[Theorem~5.2(3)]{anderson1996factorization}.

(2)
Assume that $R$ is a BFR and $M$ is a BFM, but $\idealztn{R}{M}$ is not a BFR. Then, since $R$ is a BFR, there must exist a nonzero element $(0, x) \in \idealztn{R}{M}$ which does not have the BF property. Assume that
\[
(0,x) = (a_{i,1}, x_{i,1})(a_{i, 2}, x_{i,2}) \dotsm (a_{i, n_i}, x_{i, n_i}),
\]
are factorizations of $(0, x)$ where the set $\{n_i\}_{i \in \mathbb{N}}$ is not bounded. Without loss of generality, we may assume that $0 = a_{i,1}a_{i,2} \dotsm a_{i, k_i}$ are minimal factorizations of $0$. Since $0$ is U-bounded, the set $\{ k_i \}_{i \in \mathbb{N}}$ is bounded, and so the set $\{ n_i - k_i \}_{i \in \mathbb{N}}$ must be unbounded. But then since
\[
(0,x) = (0, y_i)(a_{i,k_{i}+1}, x_{i,k_{i}+1}) \dotsm (a_{i, n_i}, x_{i, n_i})
\]
for some $y_i \in M$, we have $x = a_{i, k_{i}+1} \dotsm a_{i, n_i} y_i$, and so $M$ is not a BFM, which is a contradiction.
\end{proof}

\begin{lemma}
\label{lemmaubounded}
Let $R$ be a ring.
\begin{enumerate}
\item
If $0$ is U-bounded, then the set $\mathrm{Min}(R)$ is finite.
\item
If $R$ is reduced, then the converse also holds.
\end{enumerate}
\end{lemma}
\begin{proof}
(1)
Suppose on the contrary that $\mathrm{Min}(R)$ is infinite and let $n \in \mathbb{N}$.
Choose distinct $\mathfrak{p}' , \mathfrak{p}_1, \dotsc , \mathfrak{p}_n \in \mathrm{Min}(R)$, and using the Prime Avoidance Theorem, choose elements
\[
a_i \in \mathfrak{p}_i \setminus \bigcup_{1 \leq j \leq n , i \neq j} \mathfrak{p}_j \cup \mathfrak{p}' .
\]
In the ring $R' = R / \mathrm{Nil(R)} $, $\overline{\alpha} = \overline{a_1} \dotsm \overline{a_n} \neq \overline{0}$, since $\overline{\alpha} \notin \overline{\mathfrak{p}'}$. Since $R'$ is reduced, by \cite[Corollary~2.2]{huckaba1988commutative} and the Prime Avoidance Theorem again, $\mathrm{Ann}_{R'}(\overline{\alpha}) \nsubseteq \bigcup_{1 \leq i \leq n} \overline{\mathfrak{p}_i}$, and so there exists some
\[
\overline{\beta} \in \mathrm{Ann}_{R'}(\overline{\alpha}) \setminus \bigcup_{1 \leq i \leq n} \overline{\mathfrak{p}_i} .
\]

Therefore, $\beta a_1 \dotsm a_n \in \mathrm{Nil}(R)$, and so, for some $k \in \mathbb{N}$,
\[ {\beta}^k {a_1}^k \dotsm {a_n}^k = 0 . \]

Now any subproduct of $\beta^k {a_1}^k \dotsm {a_n}^k$ which is equal to $0$ must contain the elements $\beta, a_1 , \dotsc ,a_n$ since it belongs to every $P \in \mathrm{Min}(R)$. Hence $0$ is not U-bounded.

(2)
This is a special case of \cite[Lemma~4.16]{agargun2001factorization}. We provide a direct proof nevertheless. Since $r_1 \cdots r_n = 0$ if and only if
\[
\mathrm{Min}(R) \subseteq \bigcup_{1 \leq i \leq n} \{ \mathfrak{p} \in \mathrm{Spec}(R) \mid r_i \in \mathfrak{p} \},
\]
it follows that the factorization $0 = r_1 \cdots r_n$ can be refined to a factorization of a length less than or equal to $|\mathrm{Min}(R)|$.
\end{proof}

\begin{corollary}
Let $R$ be a reduced BFR with $|\mathrm{Min}(R)| < \infty $. Also, let $M$ be an $R$-module which is BFM. Then $\idealztn{R}{M}$ is a BFR.
\end{corollary}
\begin{proof}
It follows from Proposition \ref{BFRidealztn} and Lemma \ref{lemmaubounded}.
\end{proof}

Finally, we consider UFRs. In \cite[Corollary~9(3)]{chang2013factorization}, the authors showed that if $R$ is a field, then $\idealztn{R}{R}$ is a UFR. In the following theorem, we generalize this result. We recall that an SPIR is a local principal ideal ring with a nilpotent maximal ideal, and a semisimple module is a module that is the sum (or equivalently, direct sum) of its simple submodules. In the next theorem, we need the following result by Bouvier \cite{bouvier1974structure}: A ring $R$ is UFR if and only if 1) $R$ is a UFD, or 2) $(R, \mathfrak{m})$ is quasi-local and $\mathfrak{m}^2 = 0$, or 3) $R$ is an SPIR.

\begin{theorem}
Let $R$ be a ring and $M$ a nonzero $R$-module. The following are equivalent$:$
\begin{enumerate}
\item
$\idealztn{R}{M}$ is a UFR.
\item
$(R, \mathfrak{m})$ is quasi-local with $\mathfrak{m}^2 = 0$ and $\mathfrak{m}M = 0$.
\item
$(R, \mathfrak{m})$ is quasi-local with $\mathfrak{m}^2 = 0$ and $M$ is semisimple.
\item
 $\idealztn{R}{M}$ is pr\'{e}simplifiable and every nonzero nonunit element of \scalebox{0.97}{$\idealztn{R}{M}$} is an atom.
\end{enumerate}
\end{theorem}
\begin{proof}

$(1) \Rightarrow (2)$
The ring $\idealztn{R}{M}$ cannot be a domain, let alone a UFD.

Now, suppose that $\idealztn{R}{M}$ is quasi-local with the maximal ideal $\idealztn{I}{M}$ and $(\idealztn{I}{M})^2 = 0$. Then since $( \idealztn{I}{M})^2 = \idealztn{I^2}{IM}$, the ring $R$ is quasi-local with the maximal ideal $I$, $I^2 = 0$, and $IM = 0$.

If $\idealztn{R}{M}$ is an SPIR, then $R$ is an SPIR and the maximal ideal of $\idealztn{R}{M}$ is $\idealztn{\mathfrak{m}}{M}$, where $\mathfrak{m}$ is the unique maximal ideal of $R$. The ideal $\idealztn{0}{M}$ is principal and so $M$ is cyclic. Also, for some $\ell \in \mathbb{N}$,
\[
(\idealztn{\mathfrak{m}}{M})^\ell = \idealztn{0}{M}.
\]
If $\ell > 1$, then $M = \mathfrak{m}^{\ell-1} M$, and so by Nakayama's Lemma, $M=0$ which is a contradiction. Hence $\ell = 1$, and so $\mathfrak{m} = 0$.

$(2) \Rightarrow (3)$ If $\mathfrak{m}M=0$, then $M$ is a vector space over $R/\mathfrak{m}$, and so $M$ is semisimple.

$(3) \Rightarrow (2)$ This follows since every simple submodule of $M$ is of the form $R/\mathfrak{m}$.

$(2) \Rightarrow (4)$ It is easy to see that a quasi-local ring is pr\'{e}simplifiable, and so the first part holds. The second part follows since $(\idealztn{\mathfrak{m}}{M}) ^ 2 = 0$, and so the product of any two nonunits is $0$.

$(4) \Rightarrow (1)$ Let $x$ be a nonzero nonunit element in $\idealztn{R}{M}$. If $x=yz$, then without loss of generality, $x \sim y$, and so $y=xt$ for some $t$. Now $x = xtz$, and so $z$ is a unit since $\idealztn{R}{M}$ is pr\'{e}simplifiable. Therefore, any factorization of $x$ is of length $1$, and so $\idealztn{R}{M}$ is a UFR.
\end{proof}

\noindent \textbf{Acknowledgement.}
The authors would like to thank the referee for careful reading of the paper and very helpful comments.

\bibliographystyle{acm}
\bibliography{reference}

\end{document}